\documentclass[oneside,english]{amsart}
\usepackage[T1]{fontenc}
\usepackage[latin9]{inputenc}
\usepackage{babel}
\usepackage{amstext}
\usepackage{amsthm}
\usepackage{amssymb}
\usepackage[unicode=true,pdfusetitle,
 bookmarks=true,bookmarksnumbered=false,bookmarksopen=false,
 breaklinks=false,pdfborder={0 0 1},backref=false,colorlinks=false]
 {hyperref}

\makeatletter
\numberwithin{equation}{section}
\numberwithin{figure}{section}
\theoremstyle{plain}
\newtheorem{thm}{\protect\theoremname}
  \theoremstyle{plain}
  \newtheorem{prop}[thm]{\protect\propositionname}
  \theoremstyle{plain}
  \newtheorem{lem}[thm]{\protect\lemmaname}

\makeatother

  \providecommand{\lemmaname}{Lemma}
  \providecommand{\propositionname}{Proposition}
\providecommand{\theoremname}{Theorem}

\begin{document}

\title{\textup{Greatest lower bounds on Ricci curvature of homogeneous toric
bundles}}

\author{Yi Yao}

\thanks{Supported by China Postdoctoral Science Foundation.}
\begin{abstract}
For Fano homogeneous toric bundles, we derive a formula of the greatest
lower bound on Ricci curvature. We also give a criteria for the ampleness
of a kind of line bundles over general homogeneous toric bundles.
\end{abstract}

\email{yeeyoe@163.com}

\address{DEPARTMENT OF MATHEMATICS, SICHUAN UNIVERSITY, CHENGDU, 610064, CHINA}
\maketitle

\section{Introduction}

The existence of canonical metrics on K\"{a}hler manifolds has been a
central problem in K\"{a}hler geometry. For Fano toric manifolds, Wang
and Zhu \cite{WANGZHU} showed that there always exists K\"{a}hler-Ricci
soliton. This implies that it admits K\"{a}hler-Einstein metric if and
only if the Futaki invariant vanishes. This is also equivalent to
that the barycenter of the associated polytope coincides with the
origin.

A natural generalization of toric manifolds is the homogeneous toric
bundles. Given a generalized flag manifold $G^{\mathbb{C}}/P$ and
a toric manifold $F$ with an action by complex torus $T_{\mathbb{C}}^{m}$,
we can construct a fiber bundle via a homomorphism $\tau:P\rightarrow T_{\mathbb{C}}^{m}$,
\[
M\triangleq G^{\mathbb{C}}\times_{P,\tau}F\overset{\pi}{\rightarrow}G^{\mathbb{C}}/P.
\]
In \cite{PS1,PS2} Podest\`{a} and Spiro determined when $M$ will be
Fano, and in that case showed the existence of K\"{a}hler-Ricci solitons.

To find K\"{a}hler-Einstein metrics on a general Fano manifold $X$, we
usually consider the following equation depending on parameter $t\in[0,1]$,
i.e. continuation method along Aubin's path,
\begin{equation}
Ric(\omega_{\varphi_{t}})=t\omega_{\varphi_{t}}+(1-t)\omega_{0}.\label{eq:continuation method}
\end{equation}
It is solvable when $t$ is close to zero and solutions of the case
$t=1$ give K-E metrics.

When the K-E metrics does not exist, consider
\[
\sup\{T>0\mid(\ref{eq:continuation method})\ \textrm{is\ solvable}\ \textrm{for}\ t\in[0,T]\}.
\]
In \cite{Gabor}, Sz\'{e}kelyhidi showed that this supremum is independent
on $\omega_{0}$ and equals to the greatest lower bounds on Ricci
curvature
\[
R(X)=\sup\{t\geq0\mid\exists\omega\in c_{1}(X)\ \textrm{s.t.}\ Ric(\omega)>t\omega\}.
\]

For toric manifolds, Li \cite{ChiLi} obtained an explicit formula
for $R(X)$ in terms of the associated polytope. In addition, for
the bi-equivariant Fano compactifications of complex reductive groups,
including toric manifolds, Delcroix \cite{Delcroix} obtained a similar
formula.

One aim of this paper is to derive a similar formula for Fano homogeneous
toric bundles. Comparing to the toric case, it turns out the main
difference is that the real-reduced equation of (\ref{eq:continuation method})
has an additional term. But since this term is uniformly bounded,
we can obtain the key estimates by following the way of \cite{WANGZHU}
without essential modifications. With these estimates at hand, the
rest of the proof is similar to \cite{ChiLi}.

Now we state the result, see Section 2 for the notations. Let
\begin{equation}
\triangle_{F}=\{y\in\mathfrak{t}^{*}\mid\left\langle p_{i},y\right\rangle +1\geq0\ \textrm{for}\ i=1,\cdots,N\}\label{eq:weight polytope}
\end{equation}
be the polytope associated to Fano toric manifold $F$. Using the
dual of $d\tau|_{Z(\mathfrak{k})}:Z(\mathfrak{k})\rightarrow\mathfrak{t}$,
we denote
\[
\bigtriangleup_{M}=\left(d\tau\right)^{*}\left(\frac{1}{2\pi}\bigtriangleup_{F}\right)+I_{V}^{\vee}\subset Z(\mathfrak{k})^{*},
\]
where $I_{V}$ is defined by (\ref{eq:I_v}). With respect to the
$-\mathcal{B}$-orthogonal decomposition $\mathfrak{h}=Z(\mathfrak{k})\oplus Z(\mathfrak{k})^{\perp}$.
Let $\left(iH_{\alpha}\right)_{Z}$ be the $Z(\mathfrak{k})$-component
of $iH_{\alpha}$.
\begin{thm}
Let $M$ be a Fano homogeneous toric bundle given by data $(G^{\mathbb{C}},P,F,\tau)$.
Let
\[
P\triangleq\frac{\int_{\bigtriangleup_{M}}\vec{x}\cdot\prod_{\alpha\in R_{\mathfrak{m}}^{+}}\left(iH_{\alpha}\right)_{Z}d\mu}{\int_{\bigtriangleup_{M}}\prod_{\alpha\in R_{\mathfrak{m}}^{+}}\left(iH_{\alpha}\right)_{Z}d\mu}\in\bigtriangleup_{M}
\]
where $\left(iH_{\alpha}\right)_{Z}$ are treated as linear functions
on $Z(\mathfrak{k})^{*}$, $d\mu$ is the Lebesgue measure on the
affine subspace spanned by $\bigtriangleup_{M}$. Then
\begin{equation}
R(M)=\sup\{0\leq t<1\mid\frac{-t}{1-t}P+\frac{1}{1-t}I_{V}^{\vee}\in\bigtriangleup_{M}\}.\label{R(M) formula}
\end{equation}
\end{thm}
Another aim is to determine the ample line bundles over $M$. In \cite{PS1}
they gave a criteria for anti-canonical bundle. We only consider line
bundles in the following form
\[
\mathcal{L}=\left(G^{\mathbb{C}}\times_{P,\tau}L_{F}\right)\otimes\pi^{*}L_{\chi},
\]
where $L_{F}$ is a line bundle over $F$ with a lifted torus action,
$L_{\chi}$ is a line bundle over $G^{\mathbb{C}}/P$ given by a character
$\chi:P\rightarrow\mathbb{C}^{*}$. Let $\triangle_{L_{F}}$ be the
weight polytope associated to $L_{F}$, namely the convex hull of
the weights of action on fibers over the fixed points.
\begin{thm}
The line bundle $\mathcal{L}$ is ample if and only if $L_{F}$ is
ample and
\begin{equation}
\left(d\tau\right)^{*}\left(\frac{1}{2\pi}\triangle_{L_{F}}\right)+\frac{i}{2\pi}d\chi|_{Z(\mathfrak{k})}\subset\mathcal{C}^{\vee}\label{eq:ample condition}
\end{equation}
where $\mathcal{C}^{\vee}\subset Z(\mathfrak{k})^{*}$ is the $-\mathcal{B}$-dual
cone of the Weyl chamber $\mathcal{C}$.
\end{thm}
The main difficult to prove this is how to compute the curvature of
$\mathcal{L}$, we use a Koszul type formula (\ref{eq:Koszul}).

\section{Toric Bundles over Flag Manifolds}

In this section we follow \cite{PS1}. The only new observation is
Proposition \ref{prop:new observation}.

\subsection{Generalized flag manifolds}

Let $G$ be a compact semi-simple Lie group, $S$ be a subtorus. Take
$T$ is a maximal torus containing $S$. Let $K=C_{G}(S)$, then $K\supset T$.
Denote $\mathfrak{g},\mathfrak{k},\mathfrak{h}$ the Lie algebras
of $G,K,T$. Let $\mathfrak{m}$ be the orthogonal complement of $\mathfrak{k}$
with respect to the Killing form $\mathcal{B}$, which is a symmetric
and negative definite bilinear form on $\mathfrak{g}$. Then we have
$\mathfrak{g}=\mathfrak{k}\oplus\mathfrak{m}$ and $[\mathfrak{k},\mathfrak{m}]\subset\mathfrak{m}$.

Denote $G^{\mathbb{C}}$ the complexification of $G$. Since $T^{\mathbb{C}}$
is also a maximal torus of $G^{\mathbb{C}}$, we have the root decomposition
\[
\mathfrak{g}^{\mathbb{C}}=\mathfrak{h}^{\mathbb{C}}\oplus\sum_{\alpha\in R}\mathbb{C}E_{\alpha},
\]
where $R\subset i\mathfrak{h}^{*}$ is the root system and $\{E_{\alpha}\}$
are root vectors such that $H_{\alpha}=[E_{\alpha},E_{-\alpha}]$
is the $\mathcal{B}$-dual of $\alpha$. Let
\[
F_{\alpha}=\frac{1}{\sqrt{2}}(E_{\alpha}-E_{-\alpha}),\ G_{\alpha}=\frac{1}{\sqrt{2}}(E_{\alpha}+E_{-\alpha}),
\]
then $\{F_{\alpha},G_{\alpha}\}$ $\mathbb{R}$-spans $\left(\mathbb{C}E_{\alpha}\oplus\mathbb{C}E_{-\alpha}\right)\cap\mathfrak{g}$
and $[F_{\alpha},G_{\alpha}]=iH_{\alpha}$.

$K$ corresponds a partition of roots $R=R_{\mathfrak{k}}\sqcup R_{\mathfrak{m}}$
such that
\[
\mathfrak{k}^{\mathbb{C}}=\mathfrak{h}^{\mathbb{C}}\oplus\sum_{\alpha\in R_{\mathfrak{k}}}\mathbb{C}E_{\alpha},\ \mathfrak{m}^{\mathbb{C}}=\sum_{\alpha\in R_{\mathfrak{m}}}\mathbb{C}E_{\alpha}.
\]
Denote $Z(\mathfrak{k})\subset\mathfrak{h}$ the center of $\mathfrak{k}$.
Its orthogonal complement $Z(\mathfrak{k})^{\perp}$ in $\mathfrak{h}$
is spanned by $\{iH_{\alpha}\}_{\alpha\in R_{\mathfrak{k}}}$.

From these data we have a generalized flag manifold $V=G/K$. There
is a 1-to-1 correspondence between $G$-invariant complex structures
$J_{V}$ and the partitions $R_{\mathfrak{m}}=R_{\mathfrak{m}}^{+}\sqcup R_{\mathfrak{m}}^{-}$
satisfying
\begin{enumerate}
\item $R_{\mathfrak{m}}^{+}=-R_{\mathfrak{m}}^{-}$,
\item If $\alpha\in R_{\mathfrak{k}}\sqcup R_{\mathfrak{m}}^{+}$, $\beta\in R_{\mathfrak{m}}^{+}$
and $\alpha+\beta\in R$, then $\alpha+\beta\in R_{\mathfrak{m}}^{+}$.
\end{enumerate}
Identify $T_{eK}V$ with $\mathfrak{m}$, then the decomposition of
$T_{eK}V\otimes\mathbb{C}$ induced by $J_{V}$ is
\begin{equation}
\mathfrak{m}^{(1,0)}=\sum_{\alpha\in R_{\mathfrak{m}}^{+}}\mathbb{C}E_{\alpha},\ \mathfrak{m}^{(0,1)}=\sum_{\alpha\in R_{\mathfrak{m}}^{-}}\mathbb{C}E_{\alpha}.\label{eq:complex structure}
\end{equation}
It also gives a parabolic subgroup $P\subset G^{\mathbb{C}}$ with
Lie algebra $\mathfrak{p}=\mathfrak{k}^{\mathbb{C}}+\mathfrak{m}^{(0,1)}$.
Then we have a complex model $V=G^{\mathbb{C}}/P.$ Now $V$ is a
complex manifold with a holomorphic $G^{\mathbb{C}}$-action.

Let $\omega$ be a $G$-invariant K\"{a}hler metric on $V$, there exists
a unique $I_{\omega}\in Z(\mathfrak{k})$ such that
\begin{equation}
\omega(\hat{X},\hat{Y})|_{eK}=\mathcal{B}(I_{\omega},[X,Y])\label{eq:alge repre on base}
\end{equation}
for all $X$, $Y\in\mathfrak{g}$, where $\hat{X}$ is the induced
vector field by $G$-action. In order that $\omega$ is positive,
we need $I_{\omega}$ belongs to Weyl chamber
\[
\mathcal{C}=\{I\in Z(\mathfrak{k})\mid i\alpha(I)>0\ \textrm{for\ all}\ \alpha\in R_{\mathfrak{m}}^{+}\}.
\]
If we take
\begin{equation}
I_{\omega}=I_{V}\triangleq-\frac{i}{2\pi}\sum_{\alpha\in R_{\mathfrak{m}}^{+}}H_{\alpha},\label{eq:I_v}
\end{equation}
then $\omega$ is the $G$-invariant K\"{a}hler-Einstein metric.

\subsection{Toric bundles}

Let $F$ be a toric manifold with an action by complex torus $T_{\mathbb{C}}^{m}$.
Denote $T^{m}$ the real torus and $\mathfrak{t}$ its Lie algebra.

Let $\tau:P\rightarrow T_{\mathbb{C}}^{m}$ be a surjective homomorphism
which takes $K$ into $T^{m}$. Since $d\tau([X,Y])=0$ for all $X,Y\in\mathfrak{p}$,
we see that $d\tau|_{Z(\mathfrak{k})}:Z(\mathfrak{k})\rightarrow\mathfrak{t}$
is surjective. Then choose a $-\mathcal{B}$-orthonormal basis $\{Z_{i}\}_{i=1}^{m}$
for $\left(\ker d\tau|_{Z(\mathfrak{k})}\right)^{\perp}$.

The toric bundles have the following two models,
\[
M\triangleq G\times_{K,\tau}F\cong G^{\mathbb{C}}\times_{P,\tau}F.
\]
Where $G\times_{K,\tau}F=G\times F/\{(g,x)\sim(gk^{-1},\tau(k).x)\mid k\in K\}$.

Denote $\pi$ the natural projection $M\rightarrow V$ and $F_{o}\subset M$
the fiber over $eK$, which can be identified with $F$ via $x\mapsto[e,x]$
for $x\in F$. By the second model $M$ is a complex manifold. Note
that $TM|_{F_{o}}=TF_{o}\oplus\hat{\mathfrak{m}}$, then the complex
structure $J$ is the direct sum of the ones on $TF$ and $\mathfrak{m}$.
Namely, $J\hat{A}=\widehat{J_{V}A}$ for all $A\in\mathfrak{m}$.

Moreover, $M$ has a holomorphic $G^{\mathbb{C}}\times T_{\mathbb{C}}^{m}$-action
defined by
\[
(h,z).[g,x]=[hg,z.x]
\]
where $g,h\in G^{\mathbb{C}}$, $z\in T_{\mathbb{C}}^{m}$ and $x\in F$.
In the following we identify subgroups $G^{\mathbb{C}}\times\{e\}$,
$\{e\}\times T_{\mathbb{C}}^{m}$ with $G^{\mathbb{C}}$, $T_{\mathbb{C}}^{m}$.

Note that $F_{o}$ is stabilized by $P$ and $T_{\mathbb{C}}^{m}$.
Actually, these two actions on $F_{o}$ are equivalent with each other
through $\tau$. Since
\begin{equation}
(p,e).[e,x]=[p,x]=[e,\tau(p).x]=(e,\tau(p)).[e,x],\label{eq:equi of action}
\end{equation}
for all $p\in P$, $x\in F$.

\subsection{Algebraic representations}

Algebraic representation is a way to describe the invariant forms
on $G$-manifolds.

Let $\phi$ be a $G$-invariant closed two-form on $M$. According
to \cite{PS1}, there exists a map $Z_{\phi}:M\rightarrow\mathfrak{g}$,
called the algebraic representation of $\phi$, such that

(1) $Z_{\phi}(g.x)=\textrm{Ad}(g).Z_{\phi}(x)$ for all $g\in G$,
$x\in M$.

(2) For all $X,Y\in\mathfrak{g}$,
\begin{equation}
\phi(\hat{X},\hat{Y})=\mathcal{B}(Z_{\phi},[X,Y]).\label{eq:alge repre}
\end{equation}
Moreover, if $\phi$ is non-degenerate, i.e. a symplectic form, then
$Z_{\phi}^{\vee}=-\mathcal{B}(Z_{\phi},\cdot)$ is the moment map
for the $G$-action.

The following Proposition tells us when $\phi$ is also invariant
under $T^{m}$-action.
\begin{prop}
\label{prop:new observation}Let $\phi$ be a $G$-invariant and $J$-invariant
closed two-form on $M$, then $\phi$ is invariant under $T^{m}$-action
if and only if $Z_{\phi}|_{F_{o}}$ takes value in $Z(\mathfrak{k})$.
\end{prop}
\begin{proof}
Suppose that $\phi$ is $T^{m}$-invariant. Since $G$-action commutes
with $T^{m}$-action, we have $t_{*}\hat{X}=\hat{X}$ for all $t\in T^{m}$
and $X\in\mathfrak{g}$. Then by the property of $Z_{\phi}$,
\[
\mathcal{B}(Z_{\phi}(x),[X,Y])=\phi|_{x}(\hat{X},\hat{Y})=\phi|_{t.x}(t_{*}\hat{X},t_{*}\hat{Y})=\mathcal{B}(Z_{\phi}(t.x),[X,Y])
\]
for all $X,Y\in\mathfrak{g}$, $t\in T^{m}$ and $x\in M$. Since
$[\mathfrak{g},\mathfrak{g}]=\mathfrak{g}$, it follows that $Z_{\phi}(x)=Z_{\phi}(t.x)$.

In the following we assume $x\in F_{o}$. Take $k\in K$ such that
$\tau(k)=t$. By the equivalence of actions on $F_{o}$ (\ref{eq:equi of action}),
\[
Z_{\phi}(t.x)=Z_{\phi}(k.x)=\textrm{Ad}(k).Z_{\phi}(x).
\]
Thus $\textrm{Ad}(k).Z_{\phi}(x)=Z_{\phi}(x)$ for all $k\in K$.
It turns out $[\mathfrak{k},Z_{\phi}(x)]=0$. From the fact $[H,E_{\alpha}]=\alpha(H)E_{\alpha}$
for all $H\in\mathfrak{h}$, we can deduce that $Z_{\phi}(x)\in\mathfrak{k}$.
Thus $Z_{\phi}(x)\in Z(\mathfrak{k})$ as desired.

Conversely, tracing back the above arguments, we know that
\begin{equation}
\phi|_{x}(\hat{X},\hat{Y})=\phi|_{t.x}(t_{*}\hat{X},t_{*}\hat{Y})\label{eq:T-invariant}
\end{equation}
for all $X,Y\in\mathfrak{g}$, $t\in T^{m}$ and $x\in F_{o}$. On
the other hand, since $[\mathfrak{k},\mathfrak{m}]\subset\mathfrak{m}$
we have
\[
\phi|_{F_{o}}(\hat{Z}_{j},\hat{\mathfrak{m}})=\mathcal{B}(Z_{\phi}|_{F_{o}},[Z_{j},\mathfrak{m}])=0.
\]
Moreover, since $J\hat{\mathfrak{m}}=\hat{\mathfrak{m}}$ on $F_{o}$
and $\phi$ is $J$-invariant, we also have
\[
\phi|_{F_{o}}(J\hat{Z}_{j},\hat{\mathfrak{m}})=-\phi|_{F_{o}}(\hat{Z}_{j},J\hat{\mathfrak{m}})=0.
\]
Thus we have
\begin{equation}
\phi|_{F_{o}}(TF_{o},\hat{\mathfrak{m}})=0,\label{eq:curv vanish}
\end{equation}
since that $TF_{o}$ is spanned by $\{\hat{Z}_{j},J\hat{Z}_{j}\}_{j=1}^{m}$.
Then together with (\ref{eq:T-invariant}), it implies that $t^{*}\left(\phi|_{t.x}\right)=\phi|_{x}$
for all $t\in T^{m}$, $x\in F_{o}$. By using the $G$-action, we
know this holds for all $x\in M$.
\end{proof}

\section{Ampleness of Line Bundles}

In this section we discuss the ampleness of line bundles over $M$.

Let $L_{F}$ be an ample line bundle over $F$ with a lifted torus
action. Take a $T^{m}$-invariant metric $h_{F}$ on $L_{F}$, its
curvature is denoted by $0<\omega_{F}\in c_{1}(L_{F})$. Then $h_{F}$
together with the lifted action induce a moment map. In fact, let
$U\cong T_{\mathbb{C}}^{m}$ be the open orbit and $s(z)$ be an equivariant
section on $U$. Let $\varphi=-\log\left|s\right|_{h_{F}}^{2}$ which
is $T^{m}$-invariant. The basis $\{d\tau(Z_{k})\}_{k=1}^{m}$ of
$\mathfrak{t}$ gives logarithmic coordinates $\{x_{k}+i\theta_{k}\}$
on $U$. Then the moment map is given by
\begin{equation}
\mu=\sum_{k}\mu_{k}d\tau(Z_{k})^{*}:F\rightarrow\mathfrak{t}^{*},\ \mu_{k}=-\frac{1}{4\pi}\frac{\partial\varphi}{\partial x_{k}}\label{eq:moment map for toric}
\end{equation}
where $\{d\tau(Z_{k})^{*}\}$ is the dual basis. Denote $\frac{1}{2\pi}\triangle_{L_{F}}$
the image of $\mu$, which does not depend on $h_{F}$. Actually,
$\triangle_{L_{F}}$ is the convex hull of the weights of action on
fibers over the fixed points.

Let $\chi:P\rightarrow\mathbb{C}^{*}$ be a character. It gives a
line bundle
\[
L_{\chi}\triangleq G^{\mathbb{C}}\times_{P,\chi}\mathbb{C}\cong G\times_{K,\chi}\mathbb{C}
\]
 over $V$. Let $h_{\chi}$ be a $G$-invariant metric on $L_{\chi}$
defined by
\[
\left|[g,\lambda]\right|_{h_{\chi}}^{2}=\left|\lambda\right|^{2}
\]
where $g\in G$. Denote $\omega_{\chi}$ the induced curvature form.
Let $I_{\chi}\in Z(\mathfrak{k})$ such that
\[
I_{\chi}^{\vee}=-\mathcal{B}(I_{\chi},\cdot)=\frac{i}{2\pi}d\chi|_{Z(\mathfrak{k})},
\]
then we can check that $\omega_{\chi}(\hat{X},\hat{Y})|_{eK}=\mathcal{B}(I_{\chi},[X,Y]))$
for all $X,Y\in\mathfrak{g}$.

Consider the following holomorphic line bundle over $M$,
\begin{equation}
\mathcal{L}=\left(G^{\mathbb{C}}\times_{P,\tau}L_{F}\right)\otimes\pi^{*}L_{\chi}.\label{eq:line bundle}
\end{equation}
Note that its isomorphism class depends on the lifted torus action
on $L_{F}$.

It is easy to see that $\mathcal{L}$ is isomorphic to $G^{\mathbb{C}}\times_{P,\tau,\chi}L_{F}$,
where $P$ acts on $L_{F}$ by
\[
p.s=\chi(p)\left(\tau(p).s\right).
\]
Use this model we define a $G^{\mathbb{C}}\times T_{\mathbb{C}}^{m}$-action
on $\mathcal{L}$ by $(h,z).[g,s]=[hg,z.s]$, where $g,h\in G^{\mathbb{C}}$,
$s\in L_{F}$.

Now let us first consider $\mathcal{L}_{1}\triangleq G^{\mathbb{C}}\times_{P,\tau}L_{F}\cong G\times_{K,\tau}L_{F}$.
There is a natural metric $h_{1}$ on it defined by
\begin{equation}
\left|[g,s]\right|_{h_{1}}^{2}=\left|s\right|_{h_{F}}^{2}.\label{eq:metric L1}
\end{equation}
There is also a $G^{\mathbb{C}}\times T_{\mathbb{C}}^{m}$-action
on $\mathcal{L}_{1}$ and this metric is invariant under $G\times T^{m}$-action.
Let $\omega_{1}$ be the induced curvature form. Since the invariance,
we only consider its restriction on $TM|_{F_{o}}$. By (\ref{eq:curv vanish})
we have
\begin{equation}
\omega_{1}|_{TF_{o}}=\omega_{F},\ \omega_{1}(TF_{o},\hat{\mathfrak{m}})=0.\label{eq:Curv vanish 2}
\end{equation}

In order to compute $\omega_{1}|_{\hat{\mathfrak{m}}}$ we need an
open dense subset with a product structure. Denote $N$ the kernel
of $\tau:K\rightarrow T^{m}$. Recall that $U\subset F$ is the open
orbit.
\begin{lem}
\textup{\label{prop: Identification}The open dense subset $G\times_{K,\tau}U\subset M$
is diffeomorphism to $G/N\times\mbox{\ensuremath{\mathbb{R}}}_{+}^{m}$.
And on this open set $\mathcal{L}_{1}$ can be identified with $G/N\times\mbox{\ensuremath{\mathbb{R}}}_{+}^{m}\times\mathbb{C}$.}
\end{lem}
\begin{proof}
We denote $r_{k}=\exp x_{k}$ the coordinates of $\mbox{\ensuremath{\mathbb{R}}}_{+}^{m}$.
Define $\phi:G/N\times\mbox{\ensuremath{\mathbb{R}}}_{+}^{m}\rightarrow G\times_{K,\tau}U$
by
\[
(gN,r)\mapsto[g,r].
\]
Identify $L_{F}|_{U}$ with $U\times\mathbb{C}$ via $\lambda s(z)\mapsto(z,\lambda)$,
then we define $\Phi:G/N\times\mbox{\ensuremath{\mathbb{R}}}_{+}^{m}\times\mathbb{C}\rightarrow G\times_{K,\tau}\left(L_{F}|_{U}\right)$
by
\[
(gN,r,\lambda)\mapsto[g,(r,\lambda)].
\]
Then $\phi,\Phi$ give the desired identifications.

We only verify that $\Phi$ is surjective. For any $z\in T_{\mathbb{C}}^{m}$,
let $z=rt$ where $r\in\mbox{\ensuremath{\mathbb{R}}}_{+}^{m}$, $t\in T^{m}$.
Take $k\in K$ such that $\tau(k)=t$. Since
\[
[g,(z,\lambda)]=[g,\tau(k).(r,\lambda)]=[gk,(r,\lambda)],
\]
 thus $\Phi(gkN,r,\lambda)=[g,(z,\lambda)]$. Note that $gkN$ dose
not depend on the choice of $k$.
\end{proof}
Under this identification, the $G\times T^{m}$-action on $\mathcal{L}_{1}$
becomes $(h,t).(gN,r,\lambda)=(hgkN,r,\lambda)$, where $k\in K$
is any element such that $\tau(k)=t$. And the metric (\ref{eq:metric L1})
becomes
\[
\left|(gN,r,\lambda)\right|_{h_{1}}^{2}=\left|\lambda\right|^{2}e^{-\varphi(r)}.
\]

The following result can be seen as a Koszul type formula. See \cite{Besse}
(2.134) for the origin one.
\begin{prop}
Let $(L,h)\rightarrow M$ be a holomorphic Hermitian line bundle with
curvature $\omega\in c_{1}(L)$. Suppose that there is a compact Lie
group $G$ acts holomorphicly on $L\rightarrow M$ and preserves $h$.
For $A\in\mathfrak{g}$, denote $\hat{A}$ the induced vector field
on $M$ and $\bar{A}$ the induced vector field on $L$. The complex
structure on $L$ and $M$ both are denoted by $J$. Let $H$ be a
function on $L$ defined by $H(s)=\left|s\right|_{h}^{2}$. Then for
any $A,B\in\mathfrak{g}$, $x\in M$ we have
\begin{equation}
\omega(\hat{A},\hat{B})|_{x}=\frac{1}{4\pi}J[\bar{A},\bar{B}].\log H,\label{eq:Koszul}
\end{equation}
where the derivative is taken at any $\bar{x}\in L_{x}\backslash\{0\}.$
\end{prop}
\begin{proof}
Take an open subset $U$ including $x$ and a local frame $s$ of
$L$ on it. Identify $L|_{U}$ with $U\times\mathbb{C}$ by $\lambda s(p)\mapsto(p,\lambda)$.
Let $\phi=-\log\left|s\right|_{h}^{2}$, then $2\pi\omega=i\partial\bar{\partial}\phi$.

For $p\in U$ and $X\in T_{p}M$, the horizontal lifting of $X$ w.r.t.
Chern connection is
\[
\tilde{X}|_{(p,\lambda)}=\left(X,\lambda\partial\phi(X)\right)\in T_{p}M\oplus\mathbb{C}.
\]
For $A\in\mathfrak{g}$, the induced vector field $\bar{A}$ on $L$
has the form
\begin{equation}
\bar{A}|_{(p,\lambda)}=(\hat{A},\lambda\theta_{A}),\label{eq:horizontal lifting}
\end{equation}
where $\theta_{A}$ is a complex-value function on $U$. Since the
action is holomorphic, $\theta_{A}$ is holomorphic. Moreover, by
the preservation of $h$ we can deduce that
\begin{equation}
\textrm{Re}\theta_{A}=\frac{1}{2}\hat{A}\phi.\label{eq:Re part}
\end{equation}
It follows that
\begin{eqnarray*}
\bar{A}-\widetilde{(\hat{A})}|_{(p,\lambda)} & = & \left(0,\lambda(\theta_{A}-\partial\phi(\hat{A}))\right).
\end{eqnarray*}
Denote $if_{A}=\theta_{A}-\partial\phi(\hat{A})$, since (\ref{eq:Re part})
$f_{A}=\textrm{Im}\theta_{A}+\frac{1}{2}J\hat{A}\cdot\phi$ is real.
We claim
\begin{equation}
df_{A}=2\pi\iota_{\hat{A}}\omega.\label{eq:moment map f_A}
\end{equation}

Actually, by $\hat{A}$ is holomorphic and Cartan's formula,
\[
\partial(\hat{A}\phi)=\mathcal{L}_{\hat{A}}\partial\phi=d\left(\partial\phi(\hat{A})\right)-\iota_{\hat{A}}\partial\bar{\partial}\phi.
\]
And (\ref{eq:Re part}) implies $\theta_{A}-\hat{A}\phi=-\bar{\theta}_{A}$,
then we have
\begin{eqnarray*}
idf_{A} & = & \partial\theta_{A}-d\left(\partial\phi(\hat{A})\right)\\
 & = & \partial(\theta_{A}-\hat{A}\phi)-\iota_{\hat{A}}\partial\bar{\partial}\phi=-\iota_{\hat{A}}\partial\bar{\partial}\phi
\end{eqnarray*}
as desired.

Now for $A$, $B\in\mathfrak{g}$, by (\ref{eq:horizontal lifting})
we have
\[
[\bar{A},\bar{B}]|_{(p,\lambda)}=\left([\hat{A},\hat{B}],\lambda(\hat{A}\theta_{B}-\hat{B}\theta_{A})\right).
\]
Since $\log H(p,\lambda)=\log\left|\lambda\right|^{2}-\phi(p)$, then
\[
\frac{1}{2}J[\bar{A},\bar{B}].\log H=-\hat{A}.\textrm{Im}\theta_{B}+\hat{B}.\textrm{Im}\theta_{A}-\frac{1}{2}J[\hat{A},\hat{B}].\phi.
\]
The last term
\[
J[\hat{A},\hat{B}].\phi=[J\hat{A},\hat{B}].\phi=J\hat{A}(\hat{B}.\phi)-\hat{B}(J\hat{A}.\phi).
\]
On the other hand, by (\ref{eq:moment map f_A}),
\[
2\pi\omega(\hat{A},\hat{B})=\hat{B}.f_{A}=\hat{B}.\textrm{Im}\theta_{A}+\frac{1}{2}\hat{B}(J\hat{A}.\phi).
\]
Thus
\begin{eqnarray*}
\frac{1}{2}J[\bar{A},\bar{B}].\log H & = & -\hat{A}.\textrm{Im}\theta_{B}-\frac{1}{2}J\hat{A}(\hat{B}.\phi)+2\pi\omega(\hat{A},\hat{B})\\
 & = & -2\textrm{Im}\left(\bar{\partial}\theta_{B}(\hat{A})\right)+2\pi\omega(\hat{A},\hat{B})\\
 & = & 2\pi\omega(\hat{A},\hat{B}),
\end{eqnarray*}
where we used $\hat{B}\cdot\phi=2\textrm{Re}\theta_{B}$.
\end{proof}
Now we use (\ref{eq:Koszul}) to compute the other component of $\omega_{1}$.
Denote $\Omega$ the curvature of the metric $h_{1}\otimes\pi^{*}h_{\chi}$
on $\mathcal{L}$. Obviously, $\Omega=\omega_{1}+\pi^{*}\omega_{\chi}.$
\begin{thm}
\label{thm:curva property}For $A$, $B\in\mathfrak{m}$ we have
\begin{equation}
\omega_{1}(\hat{A},\hat{B})|_{F_{o}}=\mathcal{B}\left(\sum_{i}\mu_{i}Z_{i},[A,B]\right),\label{eq:=00005BA,B=00005D formula}
\end{equation}
where the $\{\mu_{i}\}$ are given by (\ref{eq:moment map for toric}).
And the algebraic representation of $\Omega$ restricted on $F_{o}$
is equal to $\sum_{i}\mu_{i}Z_{i}+I_{\chi}$.
\end{thm}
\begin{proof}
Restrict on the open set $G\times_{K,\tau}U$, under the identification
in Lemma \ref{prop: Identification}, the induced vector field on
$\mathcal{L}_{1}=G/N\times\mbox{\ensuremath{\mathbb{R}}}_{+}^{m}\times\mathbb{C}$
is

\[
\bar{A}|_{(gN,r,\lambda)}=g_{*}\left(\overline{\textrm{Ad}(g^{-1})A}|_{(eN,r,\lambda)}\right)=g_{*}(\textrm{Ad}(g^{-1})A,0,0).
\]
Where $\textrm{Ad}(g)^{-1}A$ is considered that belongs to $T_{eN}G/N\cong\mathfrak{g}/\mathfrak{n}=\mathfrak{m}\oplus\sum_{i}\mathbb{R}\cdot Z_{i}$.
Thus
\[
\mathcal{J}[\bar{A},\bar{B}]|_{(gN,r,\lambda)}=-\mathcal{J}\overline{[A,B]}|_{(gN,r,\lambda)}=-g_{*}\mathcal{J}(\textrm{Ad}(g^{-1})[A,B],0,0),
\]
where $\mathcal{J}$ is the complex structure on $\mathcal{L}_{1}$.

Note that on
\[
T\mathcal{L}_{1}|_{(eN,r,\lambda)}=\mathfrak{m}\oplus\sum_{i}\mathbb{R}\cdot Z_{i}\oplus\sum_{i}\mathbb{R}\cdot r_{i}\frac{\partial}{\partial r_{i}}\oplus\mathbb{C},
\]
the complex structure satisfies $\mathcal{J}Z_{i}=r_{i}\frac{\partial}{\partial r_{i}}=\frac{\partial}{\partial x_{i}}$.
And
\[
\log H(gN,r,\lambda)=\log\left|\lambda\right|^{2}-\varphi(r).
\]
It turns out that to compute $\mathcal{J}[\bar{A},\bar{B}]\log H$
we only need the $Z_{i}$-components of $\textrm{Ad}(g)^{-1}[A,B]$.

Next we further restrict on $(G\times_{K,\tau}U)\cap F_{o}=K/N\times\mbox{\ensuremath{\mathbb{R}}}_{+}^{m}$,
so $g\in K$.

For $k\in K$, we have
\begin{eqnarray*}
\textrm{Ad}(k)^{-1}[A,B] & = & -\sum_{i}\mathcal{B}(\textrm{Ad}(k)^{-1}[A,B],Z_{i})Z_{i}+C\\
 & = & -\sum_{i}\mathcal{B}([A,B],Z_{i})Z_{i}+C
\end{eqnarray*}
in $\mathfrak{g}/\mathfrak{n}$, where $C\in\mathfrak{m}$ and we
used $Z_{i}\in Z(\mathfrak{k})$. Thus
\[
\mathcal{J}[\bar{A},\bar{B}]|_{(gN,r,\lambda)}=\left(-g_{*}J_{V}C,\sum_{i}\mathcal{B}([A,B],Z_{i})\frac{\partial}{\partial x_{i}},0\right).
\]
Then by (\ref{eq:Koszul})
\begin{eqnarray*}
\omega_{1}(\hat{A},\hat{B}) & = & \frac{1}{4\pi}\mathcal{J}[\bar{A},\bar{B}].\log H=-\frac{1}{4\pi}\sum_{i}\mathcal{B}([A,B],Z_{i})\frac{\partial\varphi}{\partial x_{i}},\\
 & = & \mathcal{B}\left(\sum_{i}\mu_{i}Z_{i},[A,B]\right).
\end{eqnarray*}

Combine this with (\ref{eq:Curv vanish 2}), it is easy to check that
\[
(\omega_{1}+\pi^{*}\omega_{\chi})(\hat{X},\hat{Y})=\mathcal{B}\left(\sum_{i}\mu_{i}Z_{i}+I_{\chi},[X,Y]\right)
\]
on $F_{o}$, for all $X,Y\in\mathfrak{g}$. Thus $Z_{\Omega}|_{F_{o}}=\sum_{i}\mu_{i}Z_{i}+I_{\chi}$
as desired.
\end{proof}
Once we have the algebraic representation of $\Omega$, the proof
of Theorem 2 will be same as \cite{PS1}.
\begin{proof}[Proof of Theorem 2]
``Sufficiency'' Take a metric $h_{F}$ on $L_{F}$ with positive
curvature. It gives a metric $h_{1}\otimes\pi^{*}h_{\chi}$ on $\mathcal{L}$.
We show that its curvature form $\Omega$ is positive. By the $G$-invariance,
we only check this on $F_{o}$. Since $TM|_{F_{o}}=TF_{o}\oplus\hat{\mathfrak{m}}$,
(\ref{eq:Curv vanish 2}) and the facts $[E_{\alpha},E_{\beta}]\in\mathbb{C}\cdot E_{\alpha+\beta}$,
it turns out that we only need to check $\Omega(\hat{F}_{\alpha},J\hat{F}_{\alpha})>0$
for all $\alpha\in R_{\mathfrak{m}}^{+}$.

In fact, by Theorem \ref{thm:curva property} and $[F_{\alpha},G_{\alpha}]=\sqrt{-1}H_{\alpha}$,
\begin{eqnarray}
\Omega(\hat{F}_{\alpha},J\hat{F}_{\alpha}) & = & \Omega(\hat{F}_{\alpha},\hat{G}_{\alpha})=\mathcal{B}(\sum\mu_{i}Z_{i}+I_{\chi},[F_{\alpha},G_{\alpha}])\nonumber \\
 & = & \sqrt{-1}\alpha(\sum\mu_{i}Z_{i}+I_{\chi}).\label{eq:boundness reason}
\end{eqnarray}
Then by (\ref{eq:ample condition}) and $\left(d\tau\right)^{*}\circ\mu=\sum_{i}\mu_{i}Z_{i}^{\vee}$,
we have $\Omega(\hat{F}_{\alpha},J\hat{F}_{\alpha})>0$.

``Necessity'' Since $\mathcal{L}|_{F_{o}}\cong L_{F}$, thus $L_{F}$
is ample. Then in the same way as above, we have a curvature form
$\Omega\in c_{1}(\mathcal{L})$. Since we assume that $\mathcal{L}$
is ample, there exists a metric on $\mathcal{L}$ with positive curvature.
We can assume that this metric is $G$-invariant. Thus there is a
$G$-invariant function $\phi$ such that $\Omega+i\partial\bar{\partial}\phi>0$.
By Lemma 3.1(b) in \cite{PS1}, restrict on $F_{o}$, the algebraic
representation of $i\partial\bar{\partial}\phi$ is $-\sum_{i}\phi_{x_{i}}Z_{i}$.
Hence the algebraic representation of $\Omega+i\partial\bar{\partial}\phi$
is $\sum_{i}\left(\mu_{i}-\phi_{x_{i}}\right)Z_{i}+I_{\chi}$. Then
by (\ref{eq:boundness reason}) the positivity of $\Omega+i\partial\bar{\partial}\phi$
implies
\[
\sqrt{-1}\alpha(\sum_{i}\left(\mu_{i}-\phi_{x_{i}}\right)Z_{i}+I_{\chi})>0
\]
on $F_{o}$, for all $\alpha\in R_{\mathfrak{m}}^{+}$. Since the
additional $\phi_{x_{i}}$-terms dose not change the image of moment
map $\mu$, it implies (\ref{eq:ample condition}).
\end{proof}

\section{Proof of the Theorem 1}

Now we assume that $M$ is Fano. From the short exact sequence
\[
0\rightarrow T_{v}M\rightarrow TM\stackrel{d\pi}{\rightarrow}\pi^{*}TV\rightarrow0,
\]
we have
\begin{equation}
K_{M}^{-1}=\left(G^{\mathbb{C}}\times_{P,\tau}K_{F}^{-1}\right)\otimes\pi^{*}K_{V}^{-1},\label{eq:bundle formula for K^-1}
\end{equation}
where the torus acts on $K_{F}^{-1}$ in the canonical way. In particular,
we have $K_{M}^{-1}|_{F_{o}}=K_{F_{o}}^{-1}$. Thus $F$ is also Fano.

Let $\triangle_{F}$ be the polytope associated to $F$, it has the
form (\ref{eq:weight polytope}). Take $h_{F}^{0}$ be the pullback
of the Fubini-Study metric via the map given by global sections of
$K_{F}^{-1}$. Then it induces a moment map $\mu_{0}=\sum_{i}\mu_{0i}d\tau(Z_{i})^{*}$
with image $\frac{1}{2\pi}\triangle_{L_{F}}$, where $\mu_{0i}=-\frac{1}{4\pi}\frac{\partial u_{0}}{\partial x_{i}}$
and
\begin{equation}
u_{0}=\log\sum_{\lambda\in\triangle_{F}\cap\Lambda}e^{\left\langle \cdot,\lambda\right\rangle },\label{eq:Fubini-Study}
\end{equation}
$\Lambda\subset\mathfrak{t}^{*}$ is the weight lattice.

Follow the way in the last section, $h_{F}^{0}$ induces a $G\times T^{m}$-invariant
K\"{a}hler metric $\omega_{0}$ in $c_{1}(M)$ with algebraic representation
\[
Z_{\omega_{0}}=\sum_{i}\mu_{0i}Z_{i}+I_{V}.
\]

Now consider the equation
\begin{equation}
Ric(\omega_{\varphi_{t}})=t\omega_{\varphi_{t}}+(1-t)\omega_{0}\label{eq:Aubin's path}
\end{equation}
where $\omega_{\varphi_{t}}=\omega_{0}+\frac{i}{2\pi}\partial\bar{\partial}\varphi_{t}$.
It is solvable for $t\in[0,R(M))$.

It follows from the uniqueness of the twisted K\"{a}hler-Einstein metrics
that $\{\varphi_{t}\}$ are $G\times T^{m}$-invariant. Moreover,
$\omega_{\varphi_{t}}$ can be seen as the induced K\"{a}hler metric by
$e^{-\varphi_{t}|_{F_{o}}}\cdot h_{F}^{0}$. Thus
\[
Z_{\omega_{\varphi_{t}}}=\sum_{i}\mu_{i}^{t}Z_{i}+I_{V},\ \mu_{i}^{t}=\mu_{0i}-\frac{1}{4\pi}\frac{\partial\varphi_{t}}{\partial x_{i}}.
\]
By \cite{PS1} Proposition 3.2,
\[
Z_{Ric(\omega_{\varphi_{t}})}=\sum_{i}\frac{1}{4\pi}\frac{\partial\log D}{\partial x_{i}}Z_{i}+I_{V},\ D=\det\left(-\frac{\partial\mu_{i}^{t}}{\partial x_{j}}\right)_{i,j}\cdot\prod_{\alpha\in R_{\mathfrak{m}}^{+}}\sqrt{-1}\alpha(Z_{\omega_{\varphi_{t}}}).
\]
Combine these equalities with $(\ref{eq:Aubin's path})$, it follows
that
\[
\frac{1}{4\pi}\frac{\partial\log D}{\partial x_{i}}=\mu_{0i}-\frac{t}{4\pi}\frac{\partial\varphi_{t}}{\partial x_{i}}.
\]
 Thus $\log D+u_{0}+t\varphi_{t}=C_{t}$ for some constant $C_{t}$.
Adjust $\varphi_{t}$ by adding some constant, we have
\[
\prod_{\alpha\in R_{\mathfrak{m}}^{+}}\sqrt{-1}\alpha(Z_{\omega_{\varphi_{t}}})\cdot\det(u_{t,ij})=e^{-u_{0}-t\varphi_{t}},
\]
where $u_{t}=u_{0}+\varphi_{t}$. Denote $\Gamma(x,t)=\prod_{\alpha\in R_{\mathfrak{m}}^{+}}\sqrt{-1}\alpha(Z_{\omega_{\varphi_{t}}})$
then
\[
\Gamma(x,t)\cdot\det(u_{t,ij})=e^{-(1-t)u_{0}-tu_{t}}.
\]

Since $K_{M}^{-1}$ is ample, by Theorem 2, (\ref{eq:bundle formula for K^-1})
and (\ref{eq:boundness reason}), we see that $Z_{\omega_{\varphi_{t}}}(F_{o})\subset\mathcal{C}$.
Since the image $Z_{\omega_{\varphi_{t}}}(F_{o})$ dose not depend
on $t$, there exists constant $c,C>0$ such that
\[
c\leq\Gamma(x,t)\leq C
\]
for all $x\in\mathbb{R}^{m}$ and $t<R(M)$.

With this property, follow \cite{WANGZHU} we can obtain the following
key estimates without essential modifications.

Denote $w_{t}=(1-t)u_{0}+tu_{t}$, $m_{t}=\min\{w_{t}(x)\mid x\in\mathbb{R}^{m}\}$,
and suppose that the minimum is attained at $x_{t}$.
\begin{prop}
\label{prop:WZ estimate}\cite{WANGZHU} There exists a time-independent
constant $a,C$ and $C'$, such that
\[
\left|m_{t}\right|<C,\ w_{t}(x)\geq a\left|x-x_{t}\right|-C'
\]
 for all $t<R(M)$ and $x\in\mathbb{R}^{m}$.
\end{prop}
Make a change of variables,
\begin{eqnarray*}
\int_{\mathbb{R}^{m}}\nabla u_{t}\cdot e^{-w_{t}}dx & = & \int_{\mathbb{R}^{m}}\nabla u_{t}\cdot\Gamma\cdot\det(u_{t,ij})dx\\
 & = & \int_{\triangle}y\cdot\rho_{DH}(y)dy.
\end{eqnarray*}
Where we denote $\triangle$ the image $\nabla u_{t}(\mathbb{R}^{m})$
which dose not depend on $t$, and
\[
\rho_{DH}(y)=\prod_{\alpha\in R_{\mathfrak{m}}^{+}}\sqrt{-1}\alpha\left(\sum_{i}-\frac{y_{i}}{4\pi}Z_{i}+I_{V}\right).
\]

By the divergence theorem, $\int_{\mathbb{R}^{m}}\nabla w_{t}\cdot e^{-w_{t}}dx=0$.
Then it follows that

\begin{equation}
\frac{1}{\left|\triangle\right|}\int_{\mathbb{R}^{m}}\nabla u_{0}\cdot e^{-w_{t}}dx=\frac{-t}{1-t}P_{\triangle}\triangleq\frac{-t}{1-t}\frac{1}{\left|\triangle\right|}\int_{\triangle}y\cdot\rho_{DH}(y)dy,\label{eq:barycenter identity}
\end{equation}
where $\left|\triangle\right|=\int_{\triangle}\rho_{DH}(y)dy=\int_{\mathbb{R}^{m}}e^{-w_{t}}dx$.

Note that $0\in\triangle$, we define
\[
R_{\triangle}=\sup\{0\leq t<1\mid\frac{-t}{1-t}P_{\triangle}\in\triangle\}.
\]
 Since $\nabla u_{0}(x)\in\triangle$, it follows from (\ref{eq:barycenter identity})
that $\frac{-t}{1-t}P_{\triangle}\in\triangle$ for $t<R(M)$. Thus
$R(M)\leq R_{\triangle}$.

To prove $R(M)=R_{\triangle}$, the following arguments are due to
\cite{ChiLi}.

If $R(M)=1$ we are done. In the following we assume $R(M)<1$.

It is shown in \cite{WANGZHU} that if $\left|x_{t}\right|$ is uniformly
bounded for $t\in[0,t_{0}]$, then (\ref{eq:Aubin's path})$_{t_{0}}$
is solvable. Thus there exists a sequence $\{t_{k}\}$, $t_{k}\rightarrow R(M)^{-}$,
such that $\left|x_{t_{k}}\right|\rightarrow\infty$. By passing to
a subsequence, we can assume $\nabla u_{0}(x_{t_{k}})\rightarrow\partial\triangle$.
We take an affine function $l(y)$ such that $l(\triangle)\geq0$
and $\lim_{k}l(\nabla u_{0}(x_{t_{k}}))=0$.

Denote $\overline{dx}=\frac{1}{\left|\triangle\right|}dx$. For any
$\epsilon>0$, by Proposition \ref{prop:WZ estimate}, there exists
$R_{\epsilon}>0$ which is independent of $k$ such that
\[
\int_{\mathbb{R}^{m}\backslash B_{R_{\epsilon}}(x_{t_{k}})}l\left(\nabla u_{0}\right)\cdot e^{-w_{t_{k}}}\overline{dx}<C\cdot\int_{\mathbb{R}^{m}\backslash B_{R_{\epsilon}}(x_{t_{k}})}e^{-a\left|x-x_{t}\right|}\overline{dx}<\epsilon.
\]
On the other hand, use the explicit formula of $u_{0}$ (\ref{eq:Fubini-Study}),
it is shown in \cite{ChiLi} that there exists $C>0$ which only depends
on $\triangle$ such that
\[
l\left(\nabla u_{0}(x)\right)\leq e^{CR_{\epsilon}}l\left(\nabla u_{0}(x_{t_{k}})\right)
\]
for all $x\in B_{R_{\epsilon}}(x_{t_{k}})$.

Then by (\ref{eq:barycenter identity})

\begin{eqnarray*}
l(\frac{-t_{k}}{1-t_{k}}P_{\triangle}) & = & \int_{\mathbb{R}^{m}}l\left(\nabla u_{0}\right)\cdot e^{-w_{t_{k}}}\overline{dx}\\
 & = & \int_{\mathbb{R}^{m}\backslash B_{R_{\epsilon}}(x_{t_{k}})}l\left(\nabla u_{0}\right)\cdot e^{-w_{t_{k}}}\overline{dx}+\int_{B_{R_{\epsilon}}(x_{t_{k}})}l\left(\nabla u_{0}\right)\cdot e^{-w_{t_{k}}}\overline{dx}\\
 & < & \epsilon+e^{CR_{\epsilon}}\cdot l\left(\nabla u_{0}(x_{t_{k}})\right),
\end{eqnarray*}
then let $k>N$ such that $e^{CR_{\epsilon}}\cdot l\left(\nabla u_{0}(x_{t_{k}})\right)<\epsilon$.

Hence $l(\frac{-t_{k}}{1-t_{k}}P_{\triangle})\rightarrow0$, this
implies $\frac{-t_{k}}{1-t_{k}}P_{\triangle}\rightarrow\partial\triangle$.
Thus we have $R(M)=R_{\triangle}$ as desired.

Finally, it is easy to see that $R_{\triangle}$ is same to the right
hand side of (\ref{R(M) formula}).

\section{Example}

Take $G=SU(2)$, $K=\{diag(e^{i\theta},e^{-i\theta})\mid\theta\in\mathbb{R}\}$,
$G^{\mathbb{C}}=SL(2,\mathbb{C})$, $P=SL(2,\mathbb{C})\cap\{\textrm{upper\ triangle\ matrices}\}$.
Let
\[
H_{\alpha}=\left[\begin{array}{cc}
1\\
 & -1
\end{array}\right],\ E_{\alpha}=\left[\begin{array}{cc}
0 & 1\\
0 & 0
\end{array}\right],\ E_{-\alpha}=\left[\begin{array}{cc}
0 & 0\\
1 & 0
\end{array}\right]
\]
then $\mathfrak{h}=\mathfrak{k}=Z(\mathfrak{k})=\mathbb{R}\cdot\sqrt{-1}H_{\alpha}$,
where the root $\sqrt{-1}\alpha\in\mathfrak{h}^{*}$ such that $\alpha(H_{\alpha})=2$.
The Killing form such that $\mathcal{B}(H_{\alpha},H_{\alpha})=2$.
The root decomposition is
\[
\mathfrak{g}^{\mathbb{C}}=\mathfrak{h}^{\mathbb{C}}\oplus\mathbb{C}\cdot E_{\alpha}\oplus\mathbb{C}\cdot E_{-\alpha},
\]
 where $R=R_{\mathfrak{m}}=\{\pm\alpha\}$, $R_{\mathfrak{m}}^{+}=\{-\alpha\}$.
Thus by (\ref{eq:I_v}), $I_{V}=\frac{\sqrt{-1}}{2\pi}H_{\alpha}$.
$G/K=\mathbb{CP}^{1}$ and $I_{V}$ gives the Fubini-Study metric.

Take $F=\mathbb{CP}^{1}$, the associated polytope $\triangle_{F}=[-1,1]\subset\mathfrak{t}^{*}$.
And $\tau:P\rightarrow\mathbb{C}^{*}$ be that $\tau\left(\left[\begin{array}{cc}
z & w\\
0 & z^{-1}
\end{array}\right]\right)=z$. The toric bundle is the $\mathbb{CP}^{2}$ with $1$-point blowup.

Identify $\mathfrak{h}^{*}$ with $\mathbb{R}$ via the basis $\{\left(\sqrt{-1}H_{\alpha}\right)^{\vee}\}$.
Then $I_{V}^{\vee}=\frac{1}{2\pi}$ and $\left(d\tau\right)^{*}(\frac{-1}{4\pi}\triangle_{F})=[\frac{-1}{4\pi},\frac{1}{4\pi}]$,
so $\bigtriangleup_{M}=\left(d\tau\right)^{*}(\frac{-1}{4\pi}\triangle_{F})+I_{V}^{\vee}=[\frac{1}{4\pi},\frac{3}{4\pi}]$,
\[
P=\frac{\int_{\bigtriangleup_{M}}x\cdot x\ dx}{\int_{\bigtriangleup_{M}}x\ dx}=\frac{13}{24\pi}.
\]
Thus $R(M)=\sup\{0\leq t<1\mid\frac{-t}{1-t}P+\frac{1}{1-t}I_{V}^{\vee}\in\bigtriangleup_{M}\}=\frac{6}{7}$.

It coincides with the previous results in \cite{Gabor} and \cite{ChiLi}.

\end{document}